\documentclass[notitlepage, 12pt, reqno]{amsart}

\pdfoutput=1

\usepackage{graphicx}
\usepackage{amssymb}
\usepackage[foot]{amsaddr}
\usepackage{hyperref}
\usepackage{amsmath}
\usepackage{enumitem}
\usepackage{cite}
\usepackage{mathtools}
\usepackage{verbatim}
\usepackage{subcaption}

\usepackage{amsthm}

\usepackage{comment}
\usepackage{setspace}

\usepackage{tikz}
\usetikzlibrary{arrows.meta} 
\usetikzlibrary{decorations.pathmorphing} 

\usepackage{pgfplots}
\pgfplotsset{compat=newest}

\usepackage{color}
\usepackage[letterpaper,margin=1in]{geometry}

\usepackage{enumitem} 
\newlist{condenum}{enumerate}{1} 
\setlist[condenum]{label=\bfseries Condition \arabic*.,
                   ref=\arabic*, wide}

\numberwithin{equation}{section}

\newtheorem{theorem}{Theorem}[section]

\theoremstyle{plain}

\newtheorem{lemma}[theorem]{Lemma}
\newtheorem{proposition}[theorem]{Proposition}

\theoremstyle{definition}

\newtheorem{conjecture}[theorem]{Conjecture}

\newcommand{\parts}[1]{\mathfrak{#1}} 
\newcommand{\size}[1]{\lVert \parts{#1} \rVert} 
\newcommand{\length}[1]{| \parts{#1} |} 
\newcommand{\lengthn}[2]{| \parts{#1}_{#2} |} 

\newcommand{\defn}{\coloneqq} 

\let\oldleft\left
\let\oldright\right
\renewcommand{\left}{\mathopen{}\mathclose\bgroup\oldleft}
\renewcommand{\right}{\aftergroup\egroup\oldright}

\begin{document}

\title{On the moments of the Ulam-Kac adder}

\author{Gage Bonner}
\email{gbonner@wisc.edu}

\address{Department of Mathematics, University of Wisconsin --
	Madison, WI 53706, USA}


\date{\today}

\begin{abstract}
Let $\{U(n)\}_{n \geq 0}$ be a sequence of independent random variables such that $U(n)$ is distributed uniformly on $\{0, 1, 2 \dots n\}$. The Ulam-Kac adder is the history-dependent random sequence defined by $X_{n + 1} = X_{n} + X_{U(n)}$ with the initial condition $X_0 = 1$. We show that for each $m \geq 1$, it holds that $\log E[X_n^m]/\sqrt{n}$ approaches a constant $c_m$ as $n \to \infty$. Loose bounds are provided for the constants $c_m$.
\end{abstract}


\maketitle


\section{Introduction}
\label{sec:introduction}

History-dependent random sequences are sequences in which the distribution of the outcome of the $n^\text{th}$ step depends non-trivially on the entire history of the process. Such sequences have numerous applications in sciences such as in polymer physics where basic models of confined polymer chains in small enclosures consist of lattice sites which are occupied sequentially. The points evolve according to a self-avoiding random walk \cite{domb1969, freed1981polymers, macdonald2000self} and the investigation of the analytical properties of such processes remains an active area of research \cite{benito2018confined, parreno2020self}. 

The focus of this paper is the so-called Ulam-Kac adder, following the ``Ulam-Kac process'' defined in \cite{clifford2008history}. Let $\{U(n)\}_{n \geq 0}$ be a sequence of independent random variables such that $U(n)$ is distributed uniformly on $\{0, 1, 2 \dots n\}$. The Ulam-Kac adder is the history-dependent random sequence defined by
\begin{equation} \label{eq:kac-adder-def}
    X_{n + 1} = X_{n} + X_{U(n)}; \quad X_0 = 1.
\end{equation}
This sequence was first investigated by Ulam \cite{ulam1990analogies}. The path space for the first five steps of this sequence is shown in Figure~\ref{fig:kac_paths}.
\begin{figure}[ht]
\centering
\begin{tikzpicture}[node distance=2cm]
\node[] at (0mm,0mm) {$1$};

\draw[-{Latex[length=2mm, width=2mm]}, black, thick] (0mm, -2mm) to (0mm, -8mm);
\node[] at (0mm,-10mm) {$2$};

\draw[-{Latex[length=2mm, width=2mm]}, black, thick] (0mm,-12mm) to (-40mm,-18mm);
\draw[-{Latex[length=2mm, width=2mm]}, black, thick] (0mm,-12mm) to (40mm,-18mm);
\node[] at (-40mm,-20mm) {$3$};
\node[] at (40mm,-20mm) {$4$};

\draw[-{Latex[length=2mm, width=2mm]}, black, thick] (-40mm,-22mm) to (-60mm,-28mm);
\draw[-{Latex[length=2mm, width=2mm]}, black, thick] (-40mm,-22mm) to (-40mm,-28mm);
\draw[-{Latex[length=2mm, width=2mm]}, black, thick] (-40mm,-22mm) to (-20mm,-28mm);
\draw[-{Latex[length=2mm, width=2mm]}, black, thick] (40mm,-22mm) to (60mm,-28mm);
\draw[-{Latex[length=2mm, width=2mm]}, black, thick] (40mm,-22mm) to (40mm,-28mm);
\draw[-{Latex[length=2mm, width=2mm]}, black, thick] (40mm,-22mm) to (20mm,-28mm);
\node[] at (-60mm,-30mm) {$4$};
\node[] at (-40mm,-30mm) {$5$};
\node[] at (-20mm,-30mm) {$6$};
\node[] at (60mm,-30mm) {$8$};
\node[] at (40mm,-30mm) {$6$};
\node[] at (20mm,-30mm) {$5$};

\draw[-{Latex[length=2mm, width=2mm]}, black, thick] (-60mm,-32mm) to (-66mm,-38mm);
\draw[-{Latex[length=2mm, width=2mm]}, black, thick] (-60mm,-32mm) to (-62mm,-38mm);
\draw[-{Latex[length=2mm, width=2mm]}, black, thick] (-60mm,-32mm) to (-58mm,-38mm);
\draw[-{Latex[length=2mm, width=2mm]}, black, thick] (-60mm,-32mm) to (-54mm,-38mm);

\draw[-{Latex[length=2mm, width=2mm]}, black, thick] (-40mm,-32mm) to (-46mm,-38mm);
\draw[-{Latex[length=2mm, width=2mm]}, black, thick] (-40mm,-32mm) to (-42mm,-38mm);
\draw[-{Latex[length=2mm, width=2mm]}, black, thick] (-40mm,-32mm) to (-38mm,-38mm);
\draw[-{Latex[length=2mm, width=2mm]}, black, thick] (-40mm,-32mm) to (-34mm,-38mm);

\draw[-{Latex[length=2mm, width=2mm]}, black, thick] (-20mm,-32mm) to (-26mm,-38mm);
\draw[-{Latex[length=2mm, width=2mm]}, black, thick] (-20mm,-32mm) to (-22mm,-38mm);
\draw[-{Latex[length=2mm, width=2mm]}, black, thick] (-20mm,-32mm) to (-18mm,-38mm);
\draw[-{Latex[length=2mm, width=2mm]}, black, thick] (-20mm,-32mm) to (-14mm,-38mm);

\draw[-{Latex[length=2mm, width=2mm]}, black, thick] (60mm,-32mm) to (66mm,-38mm);
\draw[-{Latex[length=2mm, width=2mm]}, black, thick] (60mm,-32mm) to (62mm,-38mm);
\draw[-{Latex[length=2mm, width=2mm]}, black, thick] (60mm,-32mm) to (58mm,-38mm);
\draw[-{Latex[length=2mm, width=2mm]}, black, thick] (60mm,-32mm) to (54mm,-38mm);

\draw[-{Latex[length=2mm, width=2mm]}, black, thick] (40mm,-32mm) to (46mm,-38mm);
\draw[-{Latex[length=2mm, width=2mm]}, black, thick] (40mm,-32mm) to (42mm,-38mm);
\draw[-{Latex[length=2mm, width=2mm]}, black, thick] (40mm,-32mm) to (38mm,-38mm);
\draw[-{Latex[length=2mm, width=2mm]}, black, thick] (40mm,-32mm) to (34mm,-38mm);

\draw[-{Latex[length=2mm, width=2mm]}, black, thick] (20mm,-32mm) to (26mm,-38mm);
\draw[-{Latex[length=2mm, width=2mm]}, black, thick] (20mm,-32mm) to (22mm,-38mm);
\draw[-{Latex[length=2mm, width=2mm]}, black, thick] (20mm,-32mm) to (18mm,-38mm);
\draw[-{Latex[length=2mm, width=2mm]}, black, thick] (20mm,-32mm) to (14mm,-38mm);
\node[] at (-66mm,-40mm) {$5$};
\node[] at (-62mm,-40mm) {$6$};
\node[] at (-58mm,-40mm) {$7$};
\node[] at (-54mm,-40mm) {$8$};

\node[] at (-46mm,-40mm) {$6$};
\node[] at (-42mm,-40mm) {$7$};
\node[] at (-38mm,-40mm) {$8$};
\node[] at (-34mm,-40mm) {$10$};

\node[] at (-26mm,-40mm) {$7$};
\node[] at (-22mm,-40mm) {$8$};
\node[] at (-18mm,-40mm) {$9$};
\node[] at (-14mm,-40mm) {$12$};

\node[] at (66mm,-40mm) {$16$};
\node[] at (62mm,-40mm) {$12$};
\node[] at (58mm,-40mm) {$10$};
\node[] at (54mm,-40mm) {$9$};

\node[] at (46mm,-40mm) {$12$};
\node[] at (42mm,-40mm) {$10$};
\node[] at (38mm,-40mm) {$8$};
\node[] at (34mm,-40mm) {$7$};

\node[] at (26mm,-40mm) {$10$};
\node[] at (22mm,-40mm) {$9$};
\node[] at (18mm,-40mm) {$7$};
\node[] at (14mm,-40mm) {$6$};

\end{tikzpicture}
\caption{Illustration of the paths through the Ulam-Kac adder for the first five steps. Each of the paths are realized with equal probability.}
\label{fig:kac_paths}
\end{figure}
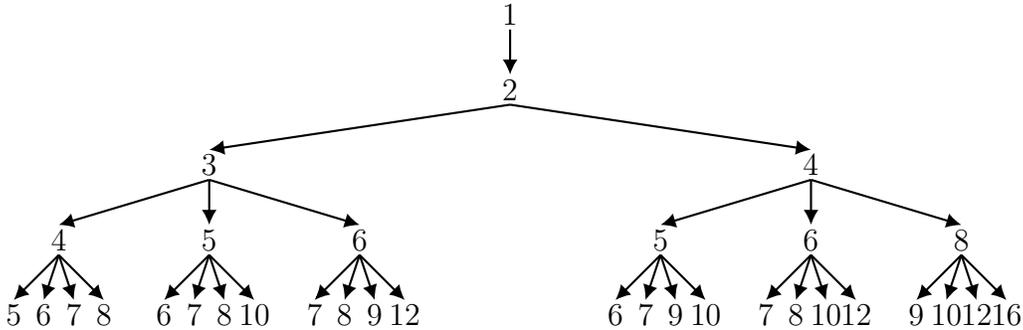
The first explicit results on the behavior of the sequence were obtained by Kac \cite{kac1989} where it was shown using generating functions and the method of steepest descent that 
\begin{align}
    \log E[X_n] &\sim 2 \sqrt{n}, \label{eq:mu1-kac}\\ 
    \log E[X_n^2] &\sim \sqrt{2 n \left(5 + \sqrt{17}\right)}. \label{eq:mu2-kac}
\end{align}
where we write $f(n) \sim g(n)$ if $\lim_{n \to \infty} f(n)/g(n) = 1$. In \cite{krasikov2004growing}, Eq.~\eqref{eq:kac-adder-def} is generalized to 
\begin{equation} \label{eq:kac-adder-def-gamma}
    X_{n + 1} = X_{n} + \gamma X_{U(n)}; \quad X_0 = 1,
\end{equation}
and the analogues of Eqs.~\eqref{eq:mu1-kac} and \eqref{eq:mu2-kac} are obtained. The Ulam-Kac adder was also considered by \cite{bennaim2002growth} in the context of extensions of random Fibonacci sequences. There, the following results are conjectured based on numerical evidence:
\begin{conjecture} \label{conj:sqrt}
There is an increasing sequence of positive constants $(c_m)_{m \geq 1}$ such that
\begin{equation}
\log E[X_n^m] \sim c_m \sqrt{n} \quad \forall m \geq 1.
\end{equation}
\end{conjecture}
\begin{conjecture} \label{conj:log-normal}
There are positive constants $\mu, \sigma$ such that
\begin{equation}
\frac{\log X_n - \mu n^{1/2}}{\sigma n^{1/4}} \to \mathcal{N}(\mu, \sigma^2) \text{ in law}.
\end{equation}
\end{conjecture}
In \cite{clifford2008history}, a Poisson-regularized continuous time analogue of the Ulam-Kac adder was introduced as an alternate method of analyzing the asymptotic behavior of moments. This was applied to the first and second moment in several generalizations and also in \cite{roitershtein2022distribution} to analyze the tail distribution of related sequences.

We mention here a connection between the Ulam-Kac adder and addition chains which does not appear to have been explicitly stated previously. An addition chain for $n$ is a sequence $(a_0, a_1, a_2 \dots a_{m - 1}, a_m)$ such that $a_0 = 1, a_m = n$ where $a_{i} = a_{j} + a_{k}$ for each $1 \leq i \leq m$ and some $0 \leq j, k < i$. An open problem in computer science is to obtain efficiently the shortest addition chain for arbitrary $n$ \cite{scholz1937aufgabe, brauer1939addition, stolarsky1969lower, schonhage1975lower}. This remains an active area of research in high-performance computing \cite{bahig2019efficient, kadir2018performance}. In the case $j = i - 1$, the chain is said to have taken a \emph{star step.} 
With reference to Figure~\ref{fig:kac_paths}, we note that each realization of the Ulam-Kac adder is addition chain consisting only of star steps. Such chains are called Brauer chains, or simply star chains. Computing $P(X_k = n)$ is equivalent to computing the number of Brauer chains for $n$ that have length $k$. The first passage times of the Ulam-Kac adder are therefore directly related to minimal length Brauer chains. 

The main result of this paper is Theorem~\ref{thm:main}, which provides an affirmative answer to Conjecture~\ref{conj:sqrt} as well as bounds on the constants $c_m$ appearing therein.  
\begin{theorem}[Main result] \label{thm:main}
For all $m \geq 1$, there exists a $c_m \in [2m,  2^{\left(\tfrac{m}{2} + 1 \right) e^{3\sqrt{m}}}]$ such that
\begin{equation}
    \log E[X_n^m] \sim c_m \sqrt{n} .
\end{equation}
\end{theorem}
The remainder of this paper is organized as follows. In Section~\ref{sec:m12}, we provide a brief review the basic method of computation for the $m = 1, 2$ moments. Then, Section~\ref{sec:proof-outline} contains our outline of the main steps of the proof. In Section~\ref{sec:m3} we apply our main result to compute $E[X_n^3]$ explicitly and compare the its value to a previous numerical estimation. Section~\ref{sec:proofs} contains all the proofs of the intermediate results required to obtain the main theorem. Finally, Section~\ref{sec:conclusion} contains our remarks on the tightness of the bounds in the main result as well as some additional conjectures related to this work. Appendix~\ref{sec:proof-extra} contains supplementary proofs.

\section{The \texorpdfstring{$m = 1, 2$}{m = 1, 2} moments}
\label{sec:m12}

The basic method for the analysis of the Ulam-Kac adder is to condition on possible values of $U(n)$ and telescoping the resulting series \cite{kac1989}, which we will now demonstrate. Defining $\mu_n \defn E[X_n]$, Eq.~\eqref{eq:kac-adder-def} gives
\begin{equation} \label{eq:first-moment-basic}
    \mu_{n + 1} = \mu_{n} + E[X_{U(n)}] =  \mu_{n} + \frac{1}{n + 1}\sum_{\ell = 0}^{n} \mu_{\ell},
\end{equation}
where we have applied $E[X_{U(n)} | U(n) = \ell] = \mu_{\ell}.$ We then take $n \to n + 1$ in Eq.~\eqref{eq:first-moment-basic} and eliminate the remaining sum to obtain
\begin{align}
    \mu_{n + 2} &=  \mu_{n +1 } + \frac{1}{n + 2}\left[mu_{n+1} + \sum_{m = 0}^{n} \mu_{m} \right] \\ 
0 &= (n + 2) \mu_{n + 2} - 2 (n + 2) \mu_{n + 1} + (n + 1) \mu_{n}, \label{eq:laguerre-recur}
\end{align}
along with the initial conditions $\mu_{0} = 1, \mu_{1} = 2$. The classical Laguerre polynomials $L_n(x)$ satisfy the recurrence relation \cite{szeg1939orthogonal}
\begin{equation}
    (2 + k)L_{k + 2}(x) - (2k + 3 - x) L_{k + 1}(x) + (k + 1) L_{k}(x) = 0, \quad L_{0} = 1, L_{1} = 1 - x.
\end{equation}
Comparing this to Eq.~\eqref{eq:laguerre-recur}, we see that they coincide when $x = -1$ and therefore 
\begin{equation}
    \mu_n = L_{n}(-1) = \sum_{\ell = 0}^{n} \binom{n}{\ell} \frac{1}{\ell!}.
\end{equation}
This directly leads to $\log \mu_n \sim 2 \sqrt{n}$ using the known asymptotics of Laguerre polynomials. See \cite{bennaim2002growth, krasikov2004growing} for alternative calculations of the same result.

We will now compute the second moment of the Ulam-Kac adder $\sigma_n \defn E[X_n^2] $ by determining an ODE satisfied by the generating function of the sequence $\{\sigma_n\}_{n \geq 0}$. By squaring Eq.~\eqref{eq:kac-adder-def}, we obtain
\begin{equation} \label{eq:kac-squared}
    X_{n + 1}^2 = X_{n}^2 + 2 X_{n} X_{U(n)} + X_{U(n)}^2 .
\end{equation}
We would like to calculate the expected value of this equation, but the cross term $X_{n} X_{U(n)}$ is not easily identified with $\sigma_n$. To obtain a closed system of recurrence relations, we define $\alpha_n \defn \sum_{\ell = 0}^{n-1} E[X_{n} X_{\ell}]$, then taking the expectation of Eq.~\eqref{eq:kac-squared} gives
\begin{equation}
    \sigma_{n + 1} = \sigma_{n} + \frac{2}{n + 1} \left(\sigma_n + \alpha_n \right) + \frac{1}{n + 1} \sum_{\ell = 0}^{n} \sigma_{\ell}. \label{eq:mu2-sig}
\end{equation}
Furthermore, by applying Eq.~\eqref{eq:kac-adder-def} to $\alpha_{n + 1}$, we obtain
\begin{equation}
\alpha_{n +1} = \sigma_{n}  + \alpha_n + \frac{1}{n + 1} \sum_{\ell = 0}^{n}  \left(\sigma_\ell + 2 \alpha_\ell \right). \label{eq:mu2-alp}
\end{equation} 
We define the generating functions $G(z) = \sum_{n = 0}^{\infty} \sigma_n z^n$ and $M(z) = \sum_{n = 0}^{\infty} \alpha_n z^n$ which have radii of convergence equal to 1 as will be shown in Section~\ref{sec:proof-outline}. Multiplying Eqs.~\eqref{eq:mu2-sig} and \eqref{eq:mu2-alp} through by $(n + 1) z^n$ and summing the resulting expressions over $n \geq 0$ gives
\begin{subequations} \label{eq:mu2-ode-system}
\begin{align}
    G' &=  (z G)' + 2(G + M) + \frac{1}{1 - z} G, \quad G(0) = 1\label{eq:G2-M2}\\ 
    M' &= (z G)' + (z M)' + \frac{1}{1 - z}(G + 2M), \quad M(0) = 0. 
\end{align}
\end{subequations}
We can eliminate $M(z)$ using Eq.~\eqref{eq:G2-M2} to obtain
\begin{equation} \label{eq:mu2-g-ode}
    (z^2 - 6z + 7) G(z) + (3z - 8) (1 - z)^2 G'(z) + (1 - z)^4 G''(z) = 0, \quad G(0) = 1, G'(0) = 4,
\end{equation}
which is the ODE we sought. Substitution of the asymptotic form $e^{S(x)}$ into Eq.~\eqref{eq:mu2-g-ode} and application of the method of dominant balance \cite{bender1999advanced} provides that $G(z) \sim \exp\left[ \left(\frac{5 + \sqrt{17}}{2} \right) \frac{1}{1 - z} \right]$ near $z = 1$ which implies that $\log \sigma_n \sim \sqrt{2(5 + \sqrt{17}) n}$.

\section{Outline of Proof}
\label{sec:proof-outline}

We will begin by constructing a closed set of recurrence relations for the $m^\text{th}$ moment. Then, we will create the system of linear differential equations for the associated generating functions. By closely studying the properties of the matrices defining this system, we are able to apply certain technical results for the asymptotic behavior of their solutions and thereby obtain Theorem~\ref{thm:main}.

Let $\mathcal{M}_{\geq 1}$ be the set which contains the empty set $\emptyset$ plus all nonempty multisets of finite cardinality on positive integers. We write an element $\parts{a} \in \mathcal{M}_{\geq 1}$ as $\parts{a} = [p_1, p_2, p_3 \dots p_n]$, where $p_i \in \mathbb{Z}_{\geq 1}$ for all $1 \leq i \leq n$ and it is allowed that $p_i = p_j$ for $i \neq j$. We call each $p_i$ a \emph{part} of $\parts{a}$ and write $\length{a} = n$ when $\parts{a}$ has $n$ parts. We reserve $\size{a}$ to denote the sum of all parts of $\length{a}$. We also recall that addition and subtraction can be defined on multisets as follows. If $\length{a}, \length{b} \in \mathcal{M}_{\geq 1}$ where $\length{a} = m$ and $\length{b} = n$ then $\parts{a} + \parts{b}$ is the multiset with $m + n$ parts created from the disjoint union of all parts of $\parts{a}$ and $\parts{b}$. Also, $\parts{a} - \parts{b}$ is the unique multisubset $\parts{c} \subseteq \parts{a}$ such that $\parts{c} + \parts{b} = \parts{a}$. Successive addition and subtraction should be applied from left to right, e.g. we have $\parts{a} + \parts{b} - \parts{c} = (\parts{a} + \parts{b}) - \parts{c}$.

Let $\mathcal{P}(n) = \{\parts{p} \in \mathcal{M}_{\geq 1} : \size{p} = n \}$ be the set of all integer partitions of $n\geq 0$ with the definition $\mathcal{P}(0) = \emptyset$. We define $\mathcal{C}^{(m)}$ as the set of all terms that can occur in our system of recurrence relations for the $m^{\text{th}}$ moment, 
\begin{equation}
    \mathcal{C}^{(m)} = \{ C_{n}^{(m)}(q; \parts{p}):  n \geq 1, 1 \leq q \leq m, \parts{p} \in \mathcal{P}(m - q) \}, \quad m \geq 1,
\end{equation}
such that each element of $\mathcal{C}^{(m)}$ is given by 
\begin{equation} \label{eq:genericC-def}
    C_n^{(m)}(q; \parts{p}) = \sum_{\alpha_1 = 0}^{n - 1} \sum_{\alpha_2 = 0}^{n - 1} \cdots \sum_{\alpha_{\length{p}} = 0}^{n - 1} E\left[X_n^q \prod_{\ell = 1}^{\length{p}} X^{p_\ell}_{\alpha_\ell} \right].
\end{equation}
Motivated by the patterns observed in Eqs.~\eqref{eq:mu2-sig} and \eqref{eq:mu2-alp}, we generally expect to obtain recurrence relations each with three types of terms: $x_n, x_n/(n + 1)$ and $\tfrac{1}{n+1} \sum_{\ell = 0}^{n} x_\ell$. We therefore introduce the following notation:
\begin{subequations}
\begin{align}
\label{eq:recurr-notation}
    \overline{x}_{n} &\defn \frac{x_n}{n + 1}, \\ 
    \hat{x}_{n} &\defn \frac{1}{n + 1} \sum_{\ell = 0}^{n} x_{\ell}.
\end{align}
\end{subequations}

We can now compute the most general recurrence relation, which is the result of the following Lemma whose proof is found in Section~\ref{sec:proof-lemma-general-recur}.
\begin{lemma} \label{lemma:general-recur}
For a fixed $m \geq 1$, each $C_{n+1}^{(m)}(q; \parts{p}) \in \mathcal{C}^{(m)}$ satisfies
\begin{subequations}
\label{eq:general-recur-result}
\begin{align}
     C_{n + 1}^{(m)}(q; \parts{p}) &= \sum_{\parts{b} \subseteq \parts{p}} C_{n}^{(m)} \left(q + \size{b} ; \parts{p} - \parts{b} \right) \label{eq:general-recur-result-top} \\ 
     &+ \sum_{\beta = 1}^{q - 1} \binom{q}{\beta} \sum_{\parts{b} \subseteq \parts{p} + [q - \beta ]} \overline{C}_{n}^{(m)} \left(\beta + \size{b} ; \parts{p}+ [q - \beta ]  - \parts{b} \right) \label{eq:general-recur-result-bar} \\
     &+ \sum_{\substack{\parts{b} \subseteq \parts{p} + [q ] \\ \parts{b} \neq \emptyset}} \hat{C}_{n}^{(m)} \left(\size{b} ; \parts{p}+ [q ] - \parts{b} \right). \label{eq:general-recur-result-hat}
\end{align}
\end{subequations}
\end{lemma}
We will show how Eq.~\eqref{eq:general-recur-result} can be written as a finite linear system of first order difference equations. This will allow us to obtain preliminary results concerning the asymptotic behavior of the system. We first collect the $C_{n}^{(m)}(q; \parts{p})$ into a vector, $h^{(m)}_n$. The following definition is important for this construction. For a fixed $m \geq 1$, the vector $h^{(m)}_n = ( C_{n}^{(m)}(q; \parts{p}) )$ is said to be in \emph{canonical ordering} if its entries are sorted first by ascending $\length{p}$ then by descending $q$. The number of entries in this vector is $\mathfrak{C}(m)$.

Now, Eq.~\eqref{eq:general-recur-result} can first be expressed generically as 
\begin{equation}
    h^{(m)}_{n + 1} = \mathbf{C}_1 h^{(m)}_{n} + \frac{1}{n + 1}\mathbf{C}_2 h^{(m)}_{n } + \frac{1}{n + 1} \sum_{\ell = 0}^{n} \mathbf{C}_3 h^{(m)}_{\ell}.
\end{equation}
Each of the matrices $\mathbf{C}_1, \mathbf{C}_2$ and $\mathbf{C}_3$ have entries which depend only on $m$ and can be extracted from Eqs.~\eqref{eq:general-recur-result-top}, \eqref{eq:general-recur-result-bar} and \eqref{eq:general-recur-result-hat}, respectively. Define
\begin{equation}
    u^{(m)}_{n + 1} = \sum_{\ell = 0}^{n} \mathbf{C}_3 h^{(m)}_{\ell} = \mathbf{C}_3 h^{(m)}_{n} + u^{(m)}_{n}.
\end{equation}
Then, we have
\begin{equation} \label{eq:linear-first-order}
    \begin{pmatrix}
        h^{(m)}_{n + 1} \\ 
        u^{(m)}_{n + 1}
    \end{pmatrix} 
    = 
    \begin{pmatrix}
        \mathbf{C}_1 + \tfrac{1}{n + 1} \mathbf{C}_2 & \tfrac{1}{n + 1} \mathbf{C}_3 \\ 
        \mathbf{C}_3 & \mathbf{I}
    \end{pmatrix}
        \begin{pmatrix}
        h^{(m)}_{n} \\ 
        u^{(m)}_{n}
    \end{pmatrix}. 
\end{equation}
We can now apply the following result due to Pituk on asymptotically constant systems of linear difference equations.
\begin{theorem}[\cite{pituk2002more}, Theorem 1] \label{thm:pituk}
Let $x = (x_n)_{n \in \mathbb{N}}$ be a solution of the difference equation $x_{n + 1} = \mathbf{L}_n x_n, n \in \mathbb{N}$. If $\lim_{n \to \infty} \mathbf{L}_n = \mathbf{L}$ for a constant matrix $\mathbf{L}$ then either $x_n = 0$ for all large $n$ or the limit
\begin{equation}
    \rho(x) = \lim_{n \to \infty} ||x_n||^{1/n}
\end{equation}
exists and is equal to the modulus of one of the eigenvalues of $\mathbf{L}$.
\end{theorem}
Examining Eq.~\eqref{eq:linear-first-order}, Theorem~\ref{thm:pituk} applies with 
\begin{equation} \label{eq:L-mat}
    \mathbf{L} =     \begin{pmatrix}
        \mathbf{C}_1  & \mathbf{0} \\ 
        \mathbf{C}_3 & \mathbf{I}
    \end{pmatrix}.
\end{equation}
We require the following proposition to compute the eigenvalues of $\mathbf{L}$.
\begin{proposition} \label{prop:Q-lower-tri}
    The matrix $\mathbf{C_1}$ is lower triangular and each entry on its diagonal is $1$.
\end{proposition}
\begin{proof}
The entries of $\mathbf{C}_1$ are controlled by Eq.~\eqref{eq:general-recur-result-top}; it implies that $(q; \parts{p})$ depends on each $(q + \size{b}; \parts{p} - \parts{b})$ for $\parts{b} \subseteq \parts{p}$. Each $(q + \size{b}; \parts{p} - \parts{b})$ comes before $(q; \parts{p})$ in the canonical ordering, which implies the lower triangular structure. There is a unique $\parts{b}$ such that $(q + \size{b}; \parts{p} - \parts{b}) = (q; \parts{p} )$, namely the empty set, and hence there are ones on the diagonal of $\mathbf{C}_1$.
\end{proof}
By Proposition~\ref{prop:Q-lower-tri}, $\mathbf{L}$ and by Eq.~\eqref{eq:L-mat}, $\mathbf{L}$ is a lower triangular matrix whose diagonal entries are all equal to 1, and hence $\mathbf{L}$ has one eigenvalue equal to 1 of multiplicity $2\mathfrak{C}(m)$. Hence, by Theorem~\ref{thm:pituk}, we have that each increasing solution $x_n$ of Eq.~\eqref{eq:linear-first-order} satisfies
\begin{equation}
\lim_{n \to \infty} ||x_n||^{1/n} = 1.
\end{equation}
We remark that a theorem of Birkhoff and Trjitzinsky (Theorem 1 of \cite{wimp1985resurrecting}) then implies that $\log(x_n) \sim n^{p}$ for a rational $p \in (0, 1)$.

To each sequence $(C_{n}^{(m)}(q; \parts{p}))_{n \geq 0}$ we associate a generating function $G^{(m)}_{q, \parts{p}}(z)$ defined by
\begin{equation} \label{eq:most-general-generating}
    G^{(m)}_{q, \parts{p}}(z) = \sum_{n = 0}^{\infty} C_{n}^{(m)}(q; \parts{p}) z^n ,
\end{equation}
which has radius of convergence equal to 1 since the growth of $(C_{n}^{(m)}(q; \parts{p}))_{n \geq 0}$ is sub-exponential. We will apply Lemma~\ref{lemma:general-recur} to obtain a system of linear differential equations satisfied by these generating functions. Therefore, we collect them into a vector $g^{(m)}(z) = (G^{(m)}_{q, \parts{p}}(z))$ in the canonical ordering. We will refer to generating functions by just their indices $(q; \parts{p})$. Also, $(q; \parts{p})_i$ will refer to the $i^\text{th}$ generating function when listed in the canonical ordering. We will similarly refer to the partition appearing in the label of the $i^\text{th}$ generating function as $\parts{p}_i$. The $m$ label will generally be omitted for brevity except where required.

Now, we multiply each term in Eq.~\eqref{eq:general-recur-result} by $(n + 1)z^n$ and sum the resulting equation over all $n \geq 0$ to transform the recurrence relations to differential equations. Recalling that the generating functions of $a_n, \overline{a}_n$ and $\hat{a}_n$ are $(z A(z))', A(z)$ and $A(z)/(1 - z)$, respectively, we obtain the following system of differential equations
\begin{equation} \label{eq:ode-system-generic}
    g'(z) = z \mathbf{C}_1 g'(z) + \left[\mathbf{C}_1 + \mathbf{C}_2 + \frac{1}{1 - z} \mathbf{C}_3 \right] g(z).
\end{equation}
The matrix $(\mathbf{I} - z \mathbf{C}_1)^{-1}$ is invertible for $0 \leq z < 1$ (by Proposition~\ref{prop:Q-lower-tri}), so we can write
\begin{equation}
    g'(z) = \mathbf{R}(z) g(z), \quad \mathbf{R}(z) = (\mathbf{I} - z \mathbf{C}_1)^{-1} \left[\mathbf{C}_1 + \mathbf{C}_2 + \frac{1}{1 - z} \mathbf{C}_3\right].  \label{eq:R-def}
\end{equation}
We make the substitution
\begin{equation} \label{eq:P-def}
    g(z) = \mathbf{P}(z) h(z), \quad \mathbf{P}(z) = \text{diag}\{(1-z)^{-\lengthn{p}{1}}, (1-z)^{-\lengthn{p}{2}}, \dots, (1-z)^{-\lengthn{p}{\mathfrak{C}(m)}}\},
\end{equation}
as well as $u = 1/(1 - z)$ which results in a new system 
\begin{equation} \label{eq:S-def}
    h'(u) = \mathbf{S}(u) h(u).
\end{equation}
In Section~\ref{sec:proof-S-properties}, we establish the following lemma which collects the relevant properties of $\mathbf{S}(u)$.
\begin{lemma}
\label{lemma:S-properties}
For $u \geq 1$, and for some $N \geq 1$ there is a sequence of matrices $\{\mathbf{M}_i \}_{0 \leq i \leq N}$ such that $\mathbf{M}_0$ is primitive with Perron-Frobenius eigenvalue bounded above by $2^{\left(\tfrac{m}{2} + 1 \right) e^{3\sqrt{m}}}$ and such that the matrix $\mathbf{S}(u)$ defined in Eq.~\eqref{eq:S-def} can be expressed as
\begin{equation}
    \mathbf{S}(u) = \sum_{i = 0}^{N} u^{-i} \mathbf{M}_i .
\end{equation}
\end{lemma}
To obtain the asymptotic behavior of $h(u)$, we will apply the following lemma, a consequence of the Hartman-Wintner theorem \cite{hartman1955asymptotic} whose proof is in Section~\ref{sec:proof-asymptotic-perron}.
\begin{lemma} \label{lemma:asyptotic-perron}
Let $d, n \geq 1$, let $(\mathbf{M}_i)_{0 \leq i \leq n}$ be $d \times d$ matrices for each $i$ and let $y(t) = (y_{i})_{1 \leq i \leq d}$ be a column vector. Consider the system of differential equations
\begin{equation} \label{eq:coppel2}
    y'(t) = \left(\sum_{\ell = 0}^{n} \mathbf{M}_\ell t^{-\ell} \right) y(t); \quad y(1) = y_0, t \geq 1.
\end{equation}
Then, if $\mathbf{M}_0$ is a non-negative, irreducible matrix with Perron-Frobenius eigenvalue $\lambda$, it holds for some $\alpha, \beta > 0$ that
\begin{equation} \label{eq:asymptotic-peron-result}
    y_{i}(t) \sim \alpha (t^\beta + o(t^\beta)) e^{\lambda t}, \quad 1 \leq i \leq d.
\end{equation}
\end{lemma}
We can now prove Theorem~\ref{thm:main}.
\begin{proof}[Proof of Theorem~\ref{thm:main}]
By Lemmas~\ref{lemma:S-properties} and \ref{lemma:asyptotic-perron}, for a given $m \geq 1$, Eq.~\eqref{eq:S-def} admits solutions $(h_i^{(m)}(u))_{1 \leq i \leq \mathfrak{C}(m)} $ such that for some $\alpha, \beta >0$,
\begin{equation}
    h_{i}^{(m)}(u) \sim \alpha(u^\beta + o(u^\beta))e^{\lambda_m u}, \quad 1 \leq i \leq \mathfrak{C}(m),
\end{equation}
where $\lambda_m$ is the Perron-Frobenius eigenvalue of the matrix $\mathbf{M}_0$ defined in Lemma~\ref{lemma:S-properties}. Hence, we have $G_{m, \emptyset}^{(m)}(z) \sim \frac{\alpha}{(1 - z)^\beta} \exp\left(\frac{\lambda_m}{1 - z} \right)$ which by the method of saddle point asymptotics \cite{flajolet2009analytic} implies that $\log E[X_n^m] \sim 2 \sqrt{\lambda_m n}$.

We will now put the required bounds on $\lambda_m$. Note that a lower bound can be obtained quickly; by Jensen's inequality, we have 
\begin{equation} \label{eq:cm-lower}
    E[X_n^m] \geq E[X_n]^m \implies \log E[X_n^m] \sim c_m \sqrt{n} \quad \text{for } c_m \geq 2m.
\end{equation} 
To obtain an upper bound, we apply Lemma~\ref{lemma:S-properties} to obtain $\lambda_m \leq 2^{\left(\tfrac{m}{2} + 1 \right) e^{3\sqrt{m}}}$ and hence that 
\begin{equation} \label{eq:cm-upper}
    \log E[X_n^m] \sim c_m \sqrt{n} \quad \text{for } c_m \leq 2^{\left(\tfrac{m}{2} + 1 \right) e^{3\sqrt{m}}}.
\end{equation}
Combining Eqs.~\eqref{eq:cm-lower} and \eqref{eq:cm-upper} completes the proof.
\end{proof}

\section{Example: \texorpdfstring{$m = 3$}{m = 3}}
\label{sec:m3}

We will now compute the exact value of the third moment of the Ulam-Kac adder, $E[X_n^3]$. This has not been explicitly computed previously, although \cite{bennaim2002growth} estimated $\log E[X_n^3] \sim 6.5 \sqrt{n}$ based on numerical evidence. When $m = 3$, there are four types of sequences in $\mathcal{C}^{(3)}$, namely 
\begin{subequations}
\label{eq:mu3-seqs}
\begin{align}
    \gamma_n &\defn C_{n}^{(3)}(3, \emptyset) = E[X_n^3], \\
    a_n &\defn C_{n}^{(3)}(2, \{1\}) = \sum_{\ell = 0}^{n-1} E[X_{n}^2 X_\ell], \\ 
    b_n &\defn C_{n}^{(3)}(1, \{2\}) = \sum_{\ell = 0}^{n-1} E[X_{n} X_\ell^2] ,  \\
    c_n &\defn C_{n}^{(3)}(1, \{1, 1\}) = \sum_{\ell = 0}^{n-1}\sum_{m = 0}^{n-1}  E[X_{n} X_\ell X_m]. 
\end{align}
\end{subequations}
Applying Lemma~\ref{lemma:general-recur}, we obtain the following system of recurrence relations. 
\begin{subequations}
\label{eq:mu3-system}
\begin{align}
    \gamma_{n+1} &= \gamma_n + 3\left( 2 \overline{\gamma}_n + \overline{a}_n + \overline{b}_n \right) + \hat{\gamma}_n,  \label{eq:mu3-gamman} \\ 
    a_{n + 1} &= \gamma_n + a_n + 2 \left(\overline{\gamma}_n + 2 \overline{a}_n + \overline{c}_n \right) + \left(\hat{\gamma}_n + \hat{a}_n + \hat{b}_n \right), \label{eq:mu3-an} \\ 
    b_{n + 1} &= \gamma_n + b_n + \left(\hat{\gamma}_n + \hat{a}_n + \hat{b}_n \right), \label{eq:mu3-bn} \\ 
    c_{n + 1} &= \gamma_n + c_n + 2a_n + \left(\hat{\gamma}_n + 3 \hat{a}_n + 3 \hat{c}_n \right). \label{eq:mu3-cn}
\end{align}
\end{subequations}
We elaborate on how Eq.~\eqref{eq:mu3-an} is obtained; the other equations follow similarly. To calculate $c_n$ we set $m = 3, q = 2$ and $\parts{p} = \{1\}$ in Eq.~\eqref{eq:general-recur-result}. The set of subsets of $\parts{p}$ is $\{\{1\}, \emptyset \}$, hence the term on the righthand side of Eq.~\eqref{eq:general-recur-result-top} is
\begin{equation}
 C_{n}^{(m)}(2 + 1; \emptyset) + C_{n}^{(m)}(2 + 0; \{1\}) = \gamma_n + a_n. 
\end{equation}
Consider now the term on Eq.~\eqref{eq:general-recur-result-bar}. Since $q = 2$, the sum over $\beta$ contains only the term $\beta = 1$, hence there will be an overall factor of $\binom{2}{1} = 2$. The sum over $\parts{b}$ is a sum over the four subsets of $\parts{p} + [1] = \{1, 1\}$, hence the entire term is
\begin{equation}
    2[\overline{C}_{n}^{(m)} \left(1 + 2 ; \emptyset \right) + 2 \overline{C}_{n}^{(m)} \left(1 + 1 ; \{1\} \right) + \overline{C}_{n}^{(m)} \left(1 + 0 ; \{1, 1\} \right)] = 2(\overline{\gamma}_n + 2\overline{a}_{n} + \overline{c}_{n}).
\end{equation}
Finally, the term on Eq.~\eqref{eq:general-recur-result-hat} is a sum over the three non-empty subsets of $\parts{p} + [2] = \{1, 2 \}$, giving
\begin{equation}
    \hat{C}_{n}^{(m)} \left(3 ; \emptyset \right) + \hat{C}_{n}^{(m)} \left(2  ; \{1\} \right) + \hat{C}_{n}^{(m)} \left(1 ; \{2\} \right) = \hat{\gamma}_n + \hat{a}_n + \hat{b}_n.
\end{equation}
Collecting all these terms gives Eq.~\eqref{eq:mu3-an} as required. Next, we construct the vector of generating functions in canonical order, $g = (G^{(3)}_{3, \emptyset},G^{(3)}_{2, \{1\}}, G^{(3)}_{1, \{2\}}, G^{(3)}_{1, \{1, 1\}} )$. Multiplying Eqs. \eqref{eq:mu3-system} by $(n + 1) z^n$ and summing over $n$ results in a system of differential equations $g'(z) = \mathbf{R}(z) g(z)$ where
\begin{equation}
\mathbf{R}(z) = \frac{1}{(1 - z)^2} \left(
\begin{array}{cccc}
 8-7 z & 3-3 z & 3-3 z & 0 \\
 4 z+\frac{1}{1-z}+3 & 6-2 z & 3 z+1 & 2-2 z \\
 6 z+\frac{1}{1-z}+1 & 3 z+1 & 2 (z+1) & 0 \\
 \frac{2-(z-1) z (2 z+11)}{(z-1)^2} & 5 z-\frac{8}{z-1}-3 & -\frac{z (3 z+5)}{z-1} & 3 z+4 \\
\end{array}
\right).
\end{equation}
We make the transformation $g(z) = \mathbf{P}(z) h(z)$ where 
\begin{equation}
    \mathbf{P} = \text{diag}\{(1-z)^{-0}, (1-z)^{-1}, (1-z)^{-1}, (1-z)^{-2}  \},
\end{equation}
as well as the substitution $u = 1/(1 - z)$, which results in the system
\begin{equation} \label{eq:mu3-de-trans}
    h(u)' = \left(\mathbf{M}_0 + \mathbf{M}_1 u^{-1} + \mathbf{M}_2 u^{-2} + \mathbf{M}_3 u^{-3} \right) h(u),
\end{equation}
where
\begin{equation}
    \mathbf{M}_0  = \left(\begin{array}{cccc}
 1 & 3 & 3 & 0 \\
 1 & 4 & 4 & 2 \\
 1 & 4 & 4 & 0 \\
 2 & 8 & 8 & 7 \\
\end{array}\right).
\end{equation}
Observe that $\mathbf{M}_0$ is a primitive matrix with Perron-Frobenius eigenvalue $\lambda^{(3)}$ given by 
\begin{subequations}
\label{eq:mu3-PF}
\begin{align}
 \lambda^{(3)} &= \max_{x \in \mathbb{R}}\{x^4 - 16 x^3 + 49 x^2 - 10 x = 0\} \\ 
 &= \frac{2}{3} \left(8+\sqrt{109} \cos \left(\frac{1}{3} \tan ^{-1}\left(\frac{6 \sqrt{22245}}{703}\right)\right)\right) \\
 &= 11.979293127704\dots 
\end{align}
\end{subequations}
By the same arguments in the proof of Theorem~\ref{thm:main}, we conclude that 
\begin{equation} \label{eq:mu3-res}
    \log E[X_n^3] \sim 2 \sqrt{\lambda^{(3)} n} \approx 6.92 \sqrt{n}.
\end{equation}
To compare this with the previous numerical estimate $\log E[X_n^3] \sim 6.5 \sqrt{n}$ from \cite{bennaim2002growth}, we reduce $g'(z) = \mathbf{R}(z) g(z)$ to a single differential equation for $\Gamma(z) \defn G_{3, \emptyset}^{(3)}(z)$. We find the following fourth-order differential equation,
\begin{equation} \label{eq:G3-diff}
    \sum_{n = 0}^{4} f_n(z) \frac{\text{d}^n \Gamma}{\text{d} z^n} = 0, \quad \Gamma(0) = 1, \Gamma'(0) = 8, \Gamma''(0) = 91, \Gamma'''(0) = 1258,
\end{equation}
where
\begin{align*}
    f_{0}(z) &= 17 - 44 z + 21 z^2 - 2 z^3 ,\\ 
    f_{1}(z) &= (1 - z) (-236 + 377 z - 173 z^2 + 22 z^3) ,\\
    f_{2}(z) &= (1 - z)^3 (174 - 156 z + 31 z^2) ,\\ 
    f_{3}(z) &= (1 - z)^5 (-27 + 11 z) ,\\ 
    f_{4}(z) &= (1 - z)^7 .
\end{align*}
We can construct a power series solution with $N$ terms to Eq.~\eqref{eq:G3-diff} to calculate the exact values of the sequence $(E[X_n^3])_{0 \leq n \leq N}$ and thus obtain numerical estimates for its asymptotic behavior. Substituting $\Gamma(z) = \sum_{n = 0}^{\infty} \gamma_n z^n$ into Eq.~\eqref{eq:G3-diff} gives the following recurrence relation for $(\gamma_n)_{n \geq 0}$,
\begin{align}
    &\sum_{i = 0}^{7} g_{i}(n) \gamma_{n + i} = 0, \quad \text{where} \notag\\ 
    &\gamma_0 = 1, \gamma_1 = 8, \gamma_2 = \tfrac{91}{2}, \gamma_3 = \tfrac{629}{3}, \gamma_4 = \tfrac{20003}{24}, \gamma_5 = \tfrac{8893}{3}, \gamma_6 = \tfrac{6953959}{720},
\end{align}
and where 
\begin{align*}
    g_0(n) &= -(1 + n)^3 (2 + n), \\
    g_1(n) &= (2 + n)^2 (54 + 40 n + 7 n^2), \\
    g_2(n) &= -(2614 + 3203 n + 1449 n^2 + 287 n^3 + 21 n^4), \\
    g_3(n) &= 10262 + 10108 n + 3686 n^2 + 590 n^3 + 35 n^4 , \\
    g_4(n) &= -(4 + n)(4430 + 2688 n + 535 n^2 + 35 n^3) ,\\
    g_5(n) &= (4 + n) (5 + n) (738 + 251 n + 21 n^2) , \\
    g_6(n) &= -(4 + n) (5 + n) (6 + n) (48 + 7 n),\\
    g_7(n) &= (n + 4)(n + 5)(n + 6)(n + 7).
\end{align*}
Hence, the first $N$ terms in the power series can be enumerated with the help of a computer. The relative error of the $n^\text{th}$ term is shown in Figure~\ref{fig:error3} where it can be seen that the convergence of $\log[E_n^3]/\sqrt{n}$ to its limiting value is extremely slow. 
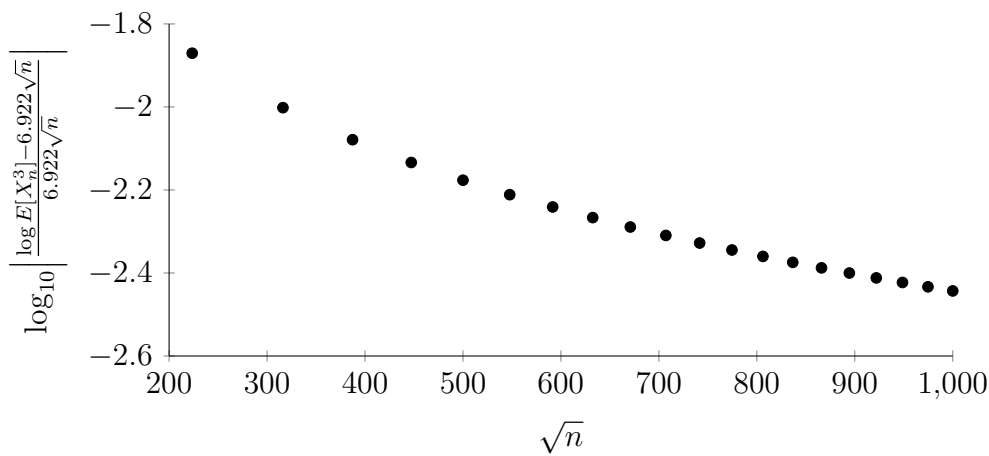
\begin{figure}[ht]
    \centering
\begin{tikzpicture}
\begin{axis}[
    height = 6cm,
    width = 12cm,
    xmin = 200,
    xmax = 1000,
    ymin = -2.6,
    ymax = -1.8,
    axis lines = left,
    xlabel = $\sqrt{n}$,
   ylabel = {$\log_{10}\left| \frac{\log E[X_n^3] - 6.922 \sqrt{n}}{6.922 \sqrt{n}} \right|$},
    xtick = {200,300,400,500,600,700,800,900,1000},
    ytick = {-1.8,-2.0,-2.2,-2.4,-2.6},
    axis line style={-}
]

\addplot[scatter, only marks, scatter src=explicit symbolic, scatter/classes = {a = {mark = circle*, black}}]
table[col sep = comma]{data/error3.csv};
\end{axis}
\end{tikzpicture}
\caption{The convergence of $E[X_n^3]/\sqrt{n}$ to its limiting value from Eq.~\eqref{eq:mu3-res}.}
\label{fig:error3}
\end{figure}
Performing a linear fit to the terms $(\sqrt{n}, \log \gamma_n)_{900 \leq n \leq 1000}$ to compare with \cite{bennaim2002growth} returns a slope of roughly $6.83$, which means that \cite{bennaim2002growth} must indeed be an underestimate. Extending our power series to $N = 10^6$ provides a slope of $6.92$, in agreement with Eq.~\eqref{eq:mu3-res}. 

\clearpage

\section{Proofs of the main steps}
\label{sec:proofs}

\subsection{Proof of Lemma~\ref{lemma:general-recur}}
\label{sec:proof-lemma-general-recur}

We take $n \to n + 1$ in Eq.~\eqref{eq:genericC-def}, then apply Eq.~\eqref{eq:kac-adder-def} raised to the power of $m$ expanded out with the binomial theorem to obtain
\begin{equation}
    C_{n + 1}^{(m)}(q; \parts{p}) = \sum_{\beta = 0}^{q} \binom{q}{\beta} \sum_{\alpha_1 = 0}^{n} \sum_{\alpha_2 = 0}^{n} \cdots \sum_{\alpha_{\length{p}} = 0}^{n} E\left[ X_{n}^\beta X_{U(n)}^{q - \beta} \prod_{\ell = 1}^{\length{p}} X^{p_\ell}_{\alpha_\ell} \right].
\end{equation}
We apply the usual method of conditioning on the value of $U(n)$ inside the expectation. We expand the result onto three lines which will be analyzed sequentially;
\begin{subequations}
\begin{align}
    C_{n + 1}^{(m)}(q; \parts{p}) &= \sum_{\alpha_1 = 0}^{n} \sum_{\alpha_2 = 0}^{n} \cdots \sum_{\alpha_{\length{p}} = 0}^{n}  E\left[ X_{n}^q \prod_{\ell = 1}^{\length{p}} X^{p_\ell}_{\alpha_\ell} \right] \label{eq:genericC-beta-q} \\
    &+ \frac{1}{n + 1} \sum_{\beta = 1}^{q - 1} \binom{q}{\beta} \sum_{\alpha_1 = 0}^{n} \sum_{\alpha_2 = 0}^{n} \cdots \sum_{\alpha_{\length{p}} = 0}^{n} \sum_{\gamma = 0}^{n} E\left[ X_{n}^\beta X_{\gamma}^{q - \beta} \prod_{\ell = 1}^{\length{p}} X^{p_\ell}_{\alpha_\ell} \right] \label{eq:genericC-beta-else} \\
    &+ \frac{1}{n + 1} \sum_{\alpha_1 = 0}^{n} \sum_{\alpha_2 = 0}^{n} \cdots \sum_{\alpha_{\length{p}} = 0}^{n} \sum_{\gamma = 0}^{n} E\left[X_{\gamma}^{q} \prod_{\ell = 1}^{\length{p}} X^{p_\ell}_{\alpha_\ell} \right]. \label{eq:genericC-beta-0}
\end{align}
\end{subequations}
First, we examine the term in Eq.~\eqref{eq:genericC-beta-q}. We must write this in terms of sums which extend to $n - 1$ rather than $n$ to compare with Eq.~\eqref{eq:genericC-def}. To see how this is done, we note that each sum $\sum_{\alpha_i  =0}^{n}$ can be expanded to two kinds of terms: one when $\alpha_i = n$ and one for $\sum_{\alpha_i  =0}^{n-1}$. Therefore, we can write all the sums as a single sum over every possible binary vector of length $\length{p}$ such that a $1$ in the $\ell^\text{th}$ entry means that $\alpha_\ell = n$ and a $0$ means $\sum_{\alpha_\ell  =0}^{n - 1}$. Consider such a binary vector $v$; the contribution to the total sum from $v$ is 
\begin{equation}
    E\left[ X_{n}^q \prod_{p_\ell \in \parts{p}} X^{p_\ell}_{\alpha_\ell} \right]_{v} = \sum_{\ell: v_\ell = 0} \sum_{\alpha_\ell = 0}^{n - 1} E\left[ X_{n}^{q + \sum_{i:v_i = 1} p_{i} } \prod_{i: v_i = 0} X^{p_\ell}_{\alpha_\ell} \right] 
\end{equation}
The sum over all such $v$ can be equivalently expressed as a sum over every possible subset of $\parts{p}$ (including the empty set) so that the term in Eq.~\eqref{eq:genericC-beta-q} can be written as
\begin{equation} \label{eq:genericC-beta-q-closed}
    \sum_{\parts{b} \subseteq \parts{p}} C_{n}^{(m)} \left(q + \size{b} ; \parts{p} - \parts{b} \right). 
\end{equation}
The term in Eq.~\eqref{eq:genericC-beta-else} is similar except for the sum over $\gamma$. By treating $q - \beta$ as a part, we can write it as 
\begin{equation} \label{eq:genericC-beta-else-closed}
    \sum_{\beta = 1}^{q - 1} \binom{q}{\beta} \sum_{\parts{b} \subseteq \parts{p} + [q - \beta ]} \overline{C}_{n}^{(m)} \left(\beta + \size{b} ; \parts{p} + [q - \beta] - \parts{b} \right),
\end{equation}
with reference to the notation of Eq.~\eqref{eq:recurr-notation}. The term in Eq.~\eqref{eq:genericC-beta-0} requires more attention. Temporarily ignoring the factor of $1/(n+1)$, we will write
\begin{subequations}
\begin{align}
    D_{n}^{(m)}(q; \parts{p}) &= \sum_{\alpha_1 = 0}^{n} \sum_{\alpha_2 = 0}^{n} \cdots \sum_{\alpha_{\length{p}} = 0}^{n} \sum_{\gamma = 0}^{n} E\left[X_{\gamma}^{q} \prod_{\ell = 1}^{\length{p}} X^{p_\ell}_{\alpha_\ell} \right] \\ 
    &=  \sum_{\substack{\parts{b} \subseteq \parts{p} + [q ] \\ \parts{b} \neq \emptyset}}C_{\ell}^{(m)} \left(\size{b} ; \parts{p} + [q] - \parts{b} \right) + D_{n - 1}^{(m)}(q; \parts{p}). \label{eq:D_sum}
\end{align}
\end{subequations}
In Eq.~\eqref{eq:D_sum}, the $\parts{b} = \emptyset$ term has been separated; note that $\parts{b} = \emptyset$ corresponds to the case where no $\alpha_i$ is equal to $n$ and hence there is no term of the form $X_n^\ell$ inside the expectation so no identification with Eq.~\eqref{eq:genericC-def} is possible. But then, Eq.~\eqref{eq:D_sum} is a linear recurrence relation for $D_{n}^{(m)}$ which can be directly solved to obtain
\begin{equation}
    D_{n}^{(m)}(q; \parts{p}) = \sum_{\ell = 0}^{n} \sum_{\substack{\parts{b} \subseteq \parts{p} + [q ] \\ \parts{b} \neq \emptyset}} C_{\ell}^{(m)} \left(\size{b} + [q] ; \parts{p} - \parts{b} \right),
\end{equation}
and so, by again referencing the notation of Eq.~\eqref{eq:recurr-notation}, the term in Eq.~\eqref{eq:genericC-beta-0} can be written as
\begin{equation}\label{eq:genericC-beta-0-closed}
    \sum_{\substack{\parts{b} \subseteq \parts{p} + [q ] \\ \parts{b} \neq \emptyset}} \hat{C}_{\ell}^{(m)} \left(\size{b} ; \parts{p} + [q] - \parts{b} \right).
\end{equation}
Collecting Eq.~\eqref{eq:genericC-beta-q-closed}, \eqref{eq:genericC-beta-else-closed} and \eqref{eq:genericC-beta-0-closed} gives Eq.~\eqref{eq:general-recur-result}.

\subsection{Proof of Lemma~\ref{lemma:S-properties}}
\label{sec:proof-S-properties}

To establish this result, we must investigate the properties of $\mathbf{C}_1, \mathbf{C}_2$ and $\mathbf{C}_3$. The sizes and entries of these matrices all depend on $m$, and the arguments in this section apply for any $m \geq 1$. As per Eq.~\eqref{eq:R-def}, the matrix $(\mathbf{I} - z \mathbf{C}_1)^{-1}$ will be important, we therefore define
\begin{equation} \label{eq:Q-def}
    \mathbf{Q} = \mathbf{I} - z \mathbf{C}_1,
\end{equation}
note that $\mathbf{Q}$ is lower triangular with diagonal entries $1 - z$ by Proposition.~\ref{prop:Q-lower-tri}. To understand $\mathbf{Q}^{-1}$, we will use the following lemma, which is a standard result from linear algebra for the inverse a triangular matrix. The proof of Proposition~\ref{prop:triangular-inverse} is in Appendix~\ref{sec:proof-extra}. 
\begin{proposition} \label{prop:triangular-inverse}
Let $X$ be an $n \times n$ lower triangular matrix,
\begin{equation}
    X = \begin{pmatrix}
    c_{1, 1} & 0 & 0 & \dots & 0 \\ 
    c_{2, 1} & c_{2, 2} & 0 & \dots & 0 \\ 
    c_{3, 1} & c_{3, 2} & c_{3, 3} & \dots & 0 \\
    \vdots & \vdots & \vdots & \ddots & \vdots\\ 
    c_{n, 1} & c_{n, 2} & c_{n, 3} & \dots & c_{n, n}
    \end{pmatrix},
\end{equation}
where $c_{i, i} > 0$ for each $1 \leq i \leq n$. Then, the entries $a_{i, j}$ of $X^{-1}$ obey
\begin{equation} \label{eq:triangular-inverse}
    a_{i, j} = \begin{cases}
    \frac{1}{c_{i, i}}\left(\delta_{i, j} - \sum_{k = 1}^{i - j} c_{i, i - k} a_{i - k, j}\right) & 1 \leq j \leq i \leq n \\ 
    0 & \text{otherwise}.
    \end{cases}
\end{equation}
\end{proposition}
We will elaborate on the interpretation of Proposition~\ref{prop:triangular-inverse} in terms of the indices $(q; \parts{p})$. In particular, consider the sum appearing in Eq.~\eqref{eq:triangular-inverse}, 
\begin{equation} \label{eq:triangular-array-sum-color}
    \sum_{k = 1}^{i - j} c_{i, i - k} a_{i - k, j}.
\end{equation}
The term $c_{i, i - k}$ relates to the dependence of $(q; \parts{p})_i$ in the canonical ordering on $(q; \parts{p})_{i - k}$. The term $a_{i - k, j}$ relates to the dependence of the $(q; \parts{p})_{i - k}$ on $(q; \parts{p})_j$, albeit in in the inverse of the matrix. The sum is shown schematically for a particular term in the $m = 4$ case in Fig.~\ref{fig:triangular-array-sum-graph}.
\begin{figure}[ht]
\centering
\begin{tikzpicture}
\node[] at (-58mm,0mm) {$(4, \emptyset)$};

\node[] at (-40mm,0mm) {$(3, \{1\})$};
\node[] at (-40mm,-7mm) {$j = 2$};
\node[] at (-40mm,-12mm) {$k = 5$};

\node[] at (-20mm,0mm) {$(2, \{2\})$};
\node[] at (-20mm,-12mm) {$k = 4$};

\node[] at (0mm,0mm) {$(1, \{3\})$};
\node[] at (0mm,-12mm) {$k = 3$};

\node[] at (20mm,0mm) {$(2, \{1,1\})$};
\node[] at (20mm,-12mm) {$k = 2$};

\node[] at (42mm,0mm) {$(1, \{2,1\})$};
\node[] at (42mm,-12mm) {$k = 1$};

\node[] at (67mm,0mm) {$(1, \{1,1,1\})$};
\node[] at (67mm,-7mm) {$i = 7$};

\draw[-{Latex[length=2mm, width=2mm]}, black, very thick] (67mm,3mm) to[out=160,in=20] (20mm,3mm);
\node[] at (42mm,10mm) {$c_{i, i - 2}$};

\draw[-{Latex[length=2mm, width=2mm]}, black, very thick] (20mm,3mm) to[out=160,in=30] (-40mm,3mm);
\node[] at (-12mm,7mm) {$a_{i - 2, j}$};

\end{tikzpicture}
\caption{In the $m = 4$ case, the terms contributing to Eq.~\eqref{eq:triangular-array-sum-color} for the calculation of $a_{6, 2}$. The sum contains five terms, one for each $1 \leq k \leq 5$. The $k = 2$ term is shown in this diagram; the arrows show how the indices $i, j$ are related to $(q; \parts{p})$ indices in the canonical ordering.}
\label{fig:triangular-array-sum-graph}
\end{figure}
We are now ready to study the inverse of $\mathbf{Q}$.
\begin{proposition}
\label{prop:Q-properties}
For the matrix $\mathbf{Q}$ defined in Eq.~\eqref{eq:Q-def}, it holds for all allowed indices $i, j$ and for $0 \leq z < 1$ that
\begin{enumerate}[label=(\roman*), leftmargin=* ]
    \item $\mathbf{Q}_{i, j}^{-1} = 0$ if and only if $\mathbf{Q}_{i, j} = 0$ \label{eq:Q-prop-1}
    \item $\mathbf{Q}_{i, j}^{-1} \geq 0$ \label{eq:Q-prop-2}
    \item If $\mathbf{Q}_{i, j}^{-1} \neq 0$, then $\mathbf{Q}_{i, j}^{-1} = (1 - z)^{-(\lengthn{p}{i} - \lengthn{p}{j} + 1)} f_{i, j}(z)$ where $f_{i, j}(z)$ is a polynomial in $z$ such that $f_{i, j} \neq 0$. \label{eq:Q-prop-3}
\end{enumerate}
\end{proposition}
\begin{proof}
By Proposition~\ref{prop:Q-lower-tri}, we can apply the result of Proposition~\ref{prop:triangular-inverse} to the entries of $\mathbf{Q}^{-1}$.

We will begin by showing \eqref{eq:Q-prop-1} by induction on the rows of $\mathbf{Q}^{-1}$. When there are two rows we have
\begin{equation} \label{eq:Q-sim-base}
\mathbf{Q}^{-1} = 
    \begin{pmatrix}
    1 - z & 0 \\ 
    -Q_{1, 2} & 1- z
    \end{pmatrix}^{-1} 
    =
     \begin{pmatrix}
    (1 - z)^{-1} & 0 \\ 
    Q_{1, 2} (1 - z)^{-2} & (1 - z)^{-1}
    \end{pmatrix}.
\end{equation}
which provides the base case for all the assertions of Lemma~\ref{lemma:S-properties} since $\mathbf{Q}_{1, 2} > 0$ via the fact that every $\parts{p}$ has the empty set as a multisubset. Proposition~\ref{prop:triangular-inverse}, shows directly that $\mathbf{Q}_{i, j}^{-1} = 0$ for $j > i$ and  $\mathbf{Q}_{i, i}^{-1} = 1/(1-z)$, hence we will restrict our attention to the case $j < i$. On the induction step, we have
\begin{equation} \label{eq:Qsum}
    \mathbf{Q}_{n, j}^{-1} =  -\frac{1}{1 - z} \sum_{k = 1}^{n - j} \mathbf{Q}_{n, n - k} \mathbf{Q}^{-1}_{n - k, j}.
\end{equation}
Suppose first that $\mathbf{Q}_{n, j} = 0$; we will examine the $k^\text{th}$ term in the sum in Eq.~\ref{eq:Qsum}.  If $\mathbf{Q}_{n, n - k} = 0$ in Eq.~\ref{eq:Qsum} then that term contributes zero to the total sum. If $\mathbf{Q}_{n, n - k} \neq 0$, it means that $\parts{p}_{n - k} \subset \parts{p}_{n}$. But then, since $\mathbf{Q}_{n, j} = 0$ we have that $\parts{p}_{j} \not\subset \parts{p}_{n}$ which implies that $\parts{p}_{j} \not \subset \parts{p}_{n - k}$ and hence that $\mathbf{Q}_{n - k, j} = 0$. By the induction hypothesis, this implies that $\mathbf{Q}^{-1}_{n - k, j} = 0$ and so every term in the sum is zero and we have $\mathbf{Q}_{n, j} = 0 \implies \mathbf{Q}^{-1}_{n, j} = 0$. Then we can apply Lemma~\ref{prop:triangular-inverse} with the roles of $\mathbf{Q}$ and $\mathbf{Q}^{-1}$ switched and the same argument above to obtain $\mathbf{Q}_{n, j}^{-1} = 0 \implies \mathbf{Q}_{n, j} = 0$. This completes the induction and the proof of \eqref{eq:Q-prop-1}. 

To establish \eqref{eq:Q-prop-2}, we perform a similar induction. Since $\mathbf{C}_1 \geq 0$, we have $\mathbf{Q}_{i, j} < 0$ for $j < i$ and $1 \leq z < 1$. The induction hypothesis provides that the sum in Eq.~\ref{eq:Qsum} contains only non-positive terms and so $\mathbf{Q}^{-1}_{n, j}$ itself is non-negative.

Finally, we establish \eqref{eq:Q-prop-3} by again using similar induction. Since $\mathbf{Q}_{n, j}^{-1} \neq 0$ by assumption, \eqref{eq:Q-prop-1} provides that $\parts{p}_{j} \subset \parts{p}_{n}$. The multiset $\parts{p}_{n} - \parts{p}_{j}$ contains at least one element, hence by the requirement of the canonical ordering that generating functions be sorted by ascending $\length{p}$, there is a $1 \leq k \leq n - j$ such that $\parts{p}_{j} \subseteq \parts{p}_{n-k} \subset \parts{p}_{n}$ and $|\parts{p}_{n-k}| = |\parts{p}_{n}| - 1$. This implies that $\mathbf{Q}_{n, n-k} > 0$ and so by the induction hypothesis, the contribution to $\mathbf{Q}_{n, j}^{-1}$ in Eq.~\ref{eq:Qsum} from this term is
\begin{equation} \label{eq:Q3-int}
    - \frac{1}{1 - z} \left[\frac{\mathbf{Q}_{n, n-k} f_{n - k, k}(z)}{(1 - z)^{(\lengthn{p}{n} - 1) - \lengthn{p}{j}  + 1}} \right] = \frac{g_{n - k, k}(z)}{(1-z)^{\lengthn{p}{n} - \lengthn{p}{j}  + 1}},
\end{equation}
where $g_{n - k, k}(z)$ is a polynomial in $z$ such that $g_{n - k, k}(1) \neq 0$. By the same argument, for each multisubset of $\parts{p}_{n} - \parts{p}_{j}$, there will be a $1 \leq k \leq n - j$ where such that $\parts{p}_{j} \subseteq \parts{p}_{n-k} \subset \parts{p}_{n}$ whose contribution to $\mathbf{Q}_{n, j}^{-1}$ has the same form as Eq.~\eqref{eq:Q3-int}. But since $|\parts{p}_{n-k}| \leq |\parts{p}_{n}| - 1$, the exponent of $(1 - z)$ is at most $\lengthn{p}{n} - \lengthn{p}{j}  + 1$. Therefore, $\mathbf{Q}_{n, j}^{-1}$ is of the required form which completes the induction and the proof.
\end{proof}
Referring to Eq.~\eqref{eq:P-def}, Proposition~\ref{prop:Q-properties} implies that we can write
\begin{equation} \label{eq:B1-def}
    \mathbf{P}^{-1}(z) \mathbf{Q}^{-1}(z) \mathbf{P}(z) = \frac{1}{1 - z} \mathbf{B}_1 + \mathbf{B}_2(z),
\end{equation}
where $\mathbf{B}_1$ is a matrix with non-negative entries arranged in a structure identical to $\mathbf{C}_1$ and $\mathbf{B}_2(z)$ is a matrix whose entries are polynomials in $(1 - z)$. Define 
\begin{equation} \label{eq:D-def}
\mathbf{D}(z) = \mathbf{P}^{-1}(z) \left[\mathbf{C}_1 + \mathbf{C}_2 + \frac{1}{1-z} \mathbf{C}_3 \right] \mathbf{P}(z),
\end{equation}
then we can write 
\begin{equation}
    \mathbf{P}^{-1}(z) \mathbf{R}(z) \mathbf{P}(z) = \left(\frac{1}{1 - z} \mathbf{B}_1 + \mathbf{B}_2(z)\right) \mathbf{D}(z)
\end{equation}
We will now study the matrix $\mathbf{D}.$ The effect of a similarity transformation by $\mathbf{P}(z)$ on a matrix with entries $(a_{i, j})$ is to transform the entries according to 
\begin{equation}
    a_{i, j} \to  a_{i, j} (1 - z)^{\lengthn{p}{i} - \lengthn{p}{j} }.
\end{equation}
From Eq.~\eqref{eq:general-recur-result}, and by the same arguments as in the proof of Proposition~\ref{prop:Q-lower-tri}, if entry $(i, j)$ in $\mathbf{C}_2$ is nonzero, then $\lengthn{p}{i} - \lengthn{p}{j} \geq -1$. Similarly, if entry $(i, j)$ in $\mathbf{C}_3$ is nonzero, then $\lengthn{p}{i} - \lengthn{p}{j} \geq 0$. Hence, we can write
\begin{equation} \label{eq:E-def}
\mathbf{D}(z) = \frac{1}{1 - z} \mathbf{E} + \mathbf{F}(z)
\end{equation}
where the entries $\mathbf{F}(z)$ are polynomials in $(1 - z)$ and the entries of $\mathbf{E}$ are obtained by restricting the sum in Eq.~\eqref{eq:general-recur-result-bar} to the term $\parts{b} = \emptyset$ and by restricting the sum in Eq.~\eqref{eq:general-recur-result-hat} to terms where $\length{b} = 1$. Explicitly, in a row indexed by $(q; \parts{p})$, the nonzero entries are either indexed by $(\beta; \parts{p} + [q - \beta])$ for some $1 \leq \beta \leq q - 1$ or by $(p; \parts{p} - [p])$ where $[p] \subseteq \parts{p} + [q]$. Note that 
\begin{equation}
    \mathbf{P}^{-1}(z) \frac{\text{d}}{\text{d} z} \mathbf{P} = \frac{1}{1-z} \text{diag}\{\lengthn{p}{1}, \lengthn{p}{2}, \dots \lengthn{p}{\mathfrak{C}(m)}\} \defn \frac{1}{1-z} \mathbf{G},
\end{equation}
hence by referring to Eqs.~\eqref{eq:P-def}, and making the substitution $u = 1/(1-z)$ we can now write 
\begin{equation}
    h'(u) = \left[\mathbf{B}_1 \mathbf{E} + \frac{1}{u}\left(\mathbf{B}_1 \mathbf{F}(u) + \mathbf{B}_2(u)\mathbf{E}  - \mathbf{G} \right) + \frac{1}{u^2} \mathbf{B}_2(u)\mathbf{F(u)} \right]  h(u),
\end{equation}
where $\mathbf{F}(u)$ and $\mathbf{B}_2$ can each be written as a finite sum of constant matrices multiplied by non-negative powers of $1/u$. Hence, to complete the proof of Lemma~\ref{lemma:S-properties} we need only establish that $\mathbf{B}_1 \mathbf{E}$ is a primitive matrix with the appropriate bounds on its Perron-Frobenius eigenvalue, which is the subject of the following proposition. Then, Lemma~\ref{lemma:S-properties} holds with $\mathbf{M}_0 = \mathbf{B}_1 \mathbf{E}$.
\begin{proposition}
    The matrix $\mathbf{B}_1 \mathbf{E}$ is primitive, where $\mathbf{B}_1$ is defined in Eq.~\eqref{eq:B1-def} and $\mathbf{E}$ is defined in Eq.~\eqref{eq:E-def}. The Perron-Frobenius eigenvalue of $\mathbf{B}_1 \mathbf{E}$ is bounded above by $m^2 2^m e^{6 \sqrt{m}}$.
\end{proposition}
\begin{proof}
    To show that $\mathbf{B}_1 \mathbf{E}$ is primitive we will show that the adjacency graph defined by $\mathbf{B}_1 \mathbf{E}$ with vertices labelled by the $(q, \parts{p})_i$ is strongly connected. Recall that the entries of $\mathbf{E}$ imply that $(q; \parts{p})$ can reach either $(\beta; \parts{p} + [q - \beta])$ for some $1 \leq \beta \leq q - 1$ or by $(p; \parts{p} - [p])$ where $[p] \subseteq \parts{p} + [q]$ and the entries of $\mathbf{B}_1$ imply that $(q; \parts{p})$ can reach $(q + \size{b}, \parts{p} - \parts{b})$ for any $\parts{b} \subseteq \parts{p}$. To demonstrate that the associated adjacency graph is connected, we will show that there is a bidirectional path between any $(q; \parts{p})$ and $(m, \emptyset)$. Suppose first that we are starting at some $(q; \parts{p})$ and want to find a path to $(m; \emptyset)$. Then, $\mathbf{E}$ provides a path from $(q; \parts{p})$ to itself and $\mathbf{B}_1$ provides a path from $(m; \emptyset)$. Now we would like to find a path starting from $(m; \emptyset)$ back to $(q; \parts{p})$. First, $\mathbf{E}$ provides a path from $(m; \emptyset)$ to $(m - p_1; [p_1])$ then $\mathbf{B}_1$ provides a path from $(m - p_1; [p_1])$ to itself. Iterating this procedure, $\mathbf{E}$ provides a path from $(m - p_1, [p_1])$ to $(m - p_1 - p_2, [p_1 , p_2])$ and $\mathbf{B_1}$ provides a path from $(m - p_1 - p_2, [p_1 , p_2])$ to itself. It's clear that that there is a $1 \leq k \leq \mathfrak{C}(m)$ such that $q = m - \sum_{\ell = 1}^{k} p_{\ell}$ and hence there will be a path from $(m; \emptyset)$ to $(q ; \parts{p})$ as required. This completes the proof that the matrix $\mathbf{B}_1 \mathbf{E}$ is primitive.

    Now we will establish the appropriate bounds on the Perron-Frobenius eigenvalue $\lambda$ of $\mathbf{B}_1 \mathbf{E}$. To do this, we will use the fact that $\lambda$ is bounded above by the maximum sum of the columns of $\mathbf{B}_1 \mathbf{E}$. By Lemma~\ref{lemma:general-recur}, the largest entry that can appear in $\mathbf{C}_1$ is $\binom{m}{\lfloor m/2 \rfloor} \leq 2^m $. Since $\mathfrak{C}(m) \leq e^{3 \sqrt{m}}$ \cite{gupta1942asymptotic}, Proposition~\ref{prop:triangular-inverse}, implies via induction that the largest entry of $\mathbf{B}_1$ is bounded above by $2^{m e^{3 \sqrt{m}}}$. Similarly, the largest entry that can appear in $\mathbf{E}$ is $\binom{m}{\lfloor m/2 \rfloor} \leq 2^m$. It follows that each entry of $\mathbf{B}_1 \mathbf{E}$ is bounded above by $2^{(m + 1) e^{3 \sqrt{m}}} e^{3 \sqrt{m}}$ and hence that $\lambda$ is bounded above by $2^{\left(\tfrac{m}{2} + 1 \right) e^{3\sqrt{m}}}$.
\end{proof}

\subsection{Proof of Lemma~\ref{lemma:asyptotic-perron}}
\label{sec:proof-asymptotic-perron}

To prove Lemma~\ref{lemma:asyptotic-perron}, we will apply two theorems from the theory of perturbations of linear systems. The following theorem is originally due to Hartman and Wintner \cite{hartman1955asymptotic} and provides to us the required asymptotic form of the solutions to our differential equation system. We quote the variation appearing in \cite{pituk1999hartman}.
\begin{theorem}[\cite{pituk1999hartman}, Theorem 1]
\label{thm:hartman}
Let $\mathbf{A}$ be a constant $d \times d$ matrix and let $y(t) = (y_{i})_{1 \leq i \leq d}$ be a column vector. Consider the system of differential equations
\begin{equation} \label{eq:coppel1}
    y'(t) = \left(\mathbf{A} + \mathbf{B}(t)\right) y(t),
\end{equation}
where $\mathbf{B}(t)$ is a continuous $d \times d$ matrix for $t \geq 0$. Let $\lambda$ be a simple eigenvalue of $\mathbf{A}$ such that no other eigenvalue of $\mathbf{A}$ has the same real part as $\lambda$ and let $v$ be the eigenvector associated to this eigenvalue. If $B(t) \in L^2(t_0, \infty)$
then, for $t_0$ large enough, Eq.~\eqref{eq:coppel1} has a solution on $[t_0, \infty)$ such that 
\begin{equation}
    y(t) = \exp\left[\mu(t - t_0) + \int_{t_0}^{t}\delta(\tau)\,\text{d}\tau \right]\left[v + o(1)\right] \quad \text{as } t \to \infty,
\end{equation}
where
\begin{equation}
    \delta(t) = \frac{(\mathbf{B}(t) v) \cdot v}{v \cdot v}.
\end{equation}
\end{theorem}
The following is Theorem 1.10.1 of \cite{eastham1989asymptotic} modified to our simpler case.
\begin{theorem}[\cite{eastham1989asymptotic}, Theorem 1.10.1]
\label{thm:eastham}
Let $\mathbf{J} = \kappa \mathbf{I} + \mu \mathbf{E}$ be an $n \times n$ matrix of Jordan type and let $\mathbf{R}(t)$ be an $n \times n$ matrix continuous on $[t_0, \infty)$. Consider the system
\begin{equation} \label{eq:jordan-sys}
    y'(t) = (J + R(t))y(t).
\end{equation}
Then, if $\mathbf{R}(t) \in L^2(t_0, \infty)$, Eq.~\eqref{eq:jordan-sys} has solutions $w_1(t), w_2(t) \dots w_n(t)$ such that for some constants $\alpha_i$,
\begin{equation}
    w_i(t) \sim t^{\alpha_i} e^{\kappa t} .
\end{equation}

\end{theorem}

We can now prove Lemma~\ref{lemma:asyptotic-perron}.
\begin{proof}[Proof of Lemma~\ref{lemma:asyptotic-perron}]
We will use the fact that the Perron-Frobenius eigenvalue of a matrix is real, simple and larger than any other eigenvalue in absolute value. We apply Theorem~\ref{thm:hartman} with $\mathbf{A} = \mathbf{M}_0$ and $\mathbf{B}(t) = \sum_{\ell = 1}^{n} M_\ell t^{-\ell}$. Since $\mathbf{B}(t)$ is at most of order $1/t$ for $t \geq 1$, all the conditions of Theorem~\ref{thm:hartman} are satisfied, and we obtain that Eq.~\eqref{eq:coppel2} has a solution of the form in Eq.~\eqref{eq:asymptotic-peron-result}.  

Next we show that there are no solutions which grow faster than this solution in the log-asymptotic sense. Transform the system Eq.~\eqref{eq:coppel2} such that $\mathbf{M}_0$ is in Jordan normal form. Then, each Jordan block defines a subsystem of equations such that the result of Theorem~\ref{thm:eastham} applies. Solutions to these subsystems grow log-asymptotically as fast as the eigenvalue associated to that Jordan block, the largest of which is the Perron-Frobenius eigenvalue of $\mathbf{M}_0$.
\end{proof}

\section{Open Problems}
\label{sec:conclusion}

Following the computation of the matrix $\mathbf{B}_1 \mathbf{E}$ in Section~\ref{sec:proof-S-properties}, the matrix $\mathbf{M}_0$ of Lemma~\ref{lemma:asyptotic-perron} can be computed exactly for small values of $m$ with the help of a computer. For example, with $m = 4$ we have
\begin{equation}
\mathbf{M}_0(m = 4) = 
\left(
\begin{array}{ccccccc}
 1 & 4 & 6 & 4 & 0 & 0 & 0 \\
 1 & 5 & 6 & 5 & 3 & 3 & 0 \\
 1 & 4 & 8 & 4 & 0 & 2 & 0 \\
 1 & 5 & 6 & 5 & 0 & 0 & 0 \\
 2 & 10 & 12 & 10 & 7 & 8 & 2 \\
 2 & 9 & 14 & 9 & 4 & 7 & 0 \\
 6 & 30 & 36 & 30 & 21 & 24 & 10 \\
\end{array}
\right).
\end{equation}
By computing the Perron-Frobenius eigenvalues of the $\mathbf{M}_0$ for each $m$, we can calculate the proportionality constants $c_m$ in $\log E[X_{n}^{m}] \sim c_m \sqrt{n}$ exactly, the results are shown in Figure~\ref{fig:momentsdata}.
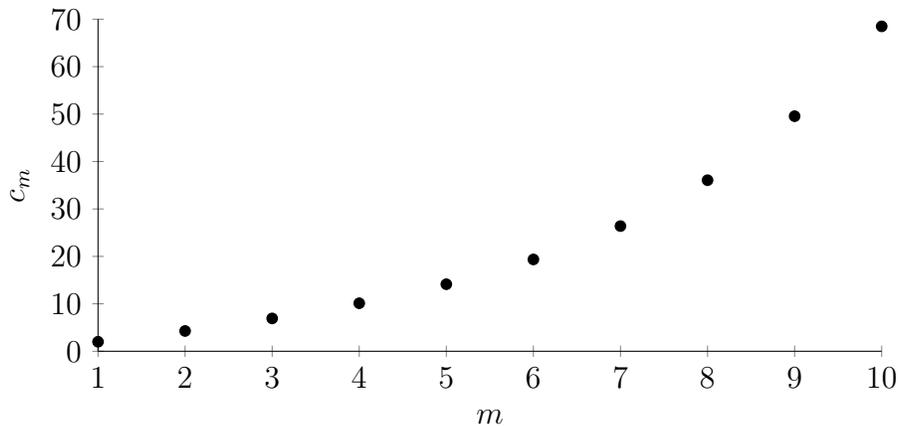
\begin{figure}[ht]
    \centering
\begin{tikzpicture}
\begin{axis}[
    height = 6cm,
    width = 12cm,
    xmin = 1,
    xmax = 10,
    ymin = 0,
    ymax = 70,
    axis lines = left,
    xlabel = $m$,
   ylabel = {$c_m$},
    xtick = {1, 2, 3, 4, 5, 6, 7, 8, 9, 10},
    ytick = {0, 10, 20, 30, 40, 50, 60, 70},
    axis line style={-}
]

\addplot[scatter, only marks, scatter src = explicit symbolic, scatter/classes = {a = {mark = circle*, black}}]
table[col sep = comma]{data/momentsdata.csv};
\end{axis}
\end{tikzpicture}
\caption{The exact values of $c_m$ for which $\log E[X_n^m] \sim c_m \sqrt{n}$ for small $m$.}
\label{fig:momentsdata}
\end{figure}
Comparing the values of $c_m$ in Figure~\ref{fig:momentsdata} to the bounds provided in Theorem~\ref{thm:main}, it's clear that the bounds in Theorem~\ref{thm:main} can be tightened significantly. Based on numerical evidence, we propose the following conjecture.
\begin{conjecture}
    Theorem~\ref{thm:main} holds with the bounds on $c_m$ replaced by the stronger bounds $\log c_m \in [0.3 m, 0.4 m + 1]$.
\end{conjecture}
Estimating the size of the Perron-Frobenius eigenvalue based on the bounds provided by row and column sums does not appear to be strong enough to obtain this. As observed in \cite{krasikov2004growing}, the generalization of the Ulam-Kac adder to Eq.~\eqref{eq:kac-adder-def-gamma} still admits first and second moments which grow log-asymptotically as $\sqrt{n}$. We therefore propose the following additional conjecture, which can likely be decided with only minor modifications to the arguments in this paper,
\begin{conjecture}
    Theorem~\ref{thm:main} holds for Eq.~\eqref{eq:kac-adder-def-gamma} with the bounds on $c_m$ replaced by $[f_1(m, \gamma), f_2(m, \gamma)]$ for suitable functions $f_1, f_2$.
\end{conjecture}
Conjecture~\ref{conj:log-normal} as well as questions related to the first passage behavior of the Ulam-Kac adder remain completely open. Many history-dependent random sequences which have been studied recently share the general feature of defining recurrence relations which contain terms of the form $X_{U(n)}$. We considered the case where $U(n)$ is a discrete uniform distribution, but generalizations to non-uniform distributions such as those mentioned in \cite{krasikov2004growing} require further investigation.  

\section*{Acknowledgements}

The author would like to thank Benedek Valk{\'o} for providing several very helpful comments on an earlier draft of this paper as well as Yu-Qiu Zhao for providing a number of references on the asymptotic theory of difference equations.

\appendix 
\section{Supplementary Material}
\label{sec:proof-extra}





\begin{proof}[Proof of Lemma~\ref{prop:triangular-inverse}]
The proof is by induction on $n$. For the base case $n = 1$, we have $X^{-1} = (1/c_{1, 1})$ and Eq.~\eqref{eq:triangular-inverse} reads $a_{1, 1} = 1/c_{1, 1}$. The induction step proceeds by Gaussian elimination of the $n^{\text{th}}$ row with the first $n - 1$ rows already reduced. We then have
\begin{equation}
    \left[
\begin{array}{ccccc|ccccc}
1 & 0 & 0 & \dots & 0 & a_{1, 1} & 0 & 0 & \dots & 0 \\
0 & 1 & 0 & \dots & 0 & a_{2, 1} & a_{2, 2} & 0 & \dots & 0 \\
0 & 0 & 1 & \dots & 0 & a_{3, 1} & a_{3, 2} & a_{3, 3} & \dots & 0 \\
\vdots & \vdots & \vdots & \ddots & \vdots & \vdots & \vdots & \vdots & \ddots & \vdots \\
c_{n, 1} & c_{n, 2} & c_{n, 3} & \dots & c_{n, n} & 0 & 0 & 0 & \dots & 1
\end{array}
\right]
\end{equation}
We see that to reduce $c_{n, j}$ for $1 \leq j \leq n - 1$, we multiply the $j^\text{th}$ row by $-c_{n, j}$ and add it to the $n^\text{th}$ row. This contributes a term $-c_{n, j} a_{j, k}$ to $a_{n, k}$ for each $1 \leq k \leq j$. Then, to reduce  $c_{n, n}$, we divide the $n^\text{th}$ row by $c_{n, n}$. Therefore,
\begin{equation}
    a_{n, j} = \frac{1}{c_{n, n}}\left(\delta_{n, j} - \sum_{k = 1}^{n - j} c_{n, n - k} a_{n - k, j}\right), \quad 1 \leq j \leq n,
\end{equation}
which completes the induction.

\end{proof}

\bibliographystyle{siam}
\bibliography{mybib.bib}

\end{document}